\documentclass[12pt]{article}

\usepackage[english]{babel}
\usepackage{amssymb}
\usepackage{amsfonts,amsmath,mathtools}
\usepackage{graphics}
\usepackage{lipsum}
\usepackage[ruled]{algorithm2e}
\usepackage{float}
\usepackage{url}
\usepackage{subfig}
\usepackage{xcolor}
\usepackage[export]{adjustbox}
\usepackage{tikz-cd}
\usepackage{titlesec}
\usepackage[all]{xy}

\titleformat{\section}
{\centering\normalfont\scshape}{\thesection.}{0.5em}{}
\titleformat{\subsection}[runin]{\normalfont}{\thesubsection.}{0.5em}{\textbf}[]

\usepackage{etoolbox}
\apptocmd{\thebibliography}{\setlength{\itemsep}{-3pt}}{}{}

\usepackage{bm}
\usepackage{enumitem}
\usepackage{faktor}
\usepackage{comment}

\usepackage[top=3cm,bottom=3cm,left=3.2cm,right=3.2cm,asymmetric]{geometry}

\newcommand\blfootnote[1]{%
	\begingroup
	\renewcommand\thefootnote{}\footnote{#1}%
	\addtocounter{footnote}{-1}%
	\endgroup
}

\usepackage{titlesec}
\titleformat{\section}
{\centering\normalfont\scshape}{\thesection.}{0.5em}{}

\newtheorem{Lem}{Lemma}[section]
\newtheorem{Prop}[Lem]{Proposition}
\newtheorem{Cor}[Lem]{Corollary}
\newtheorem{Thm}[Lem]{Theorem}

\newtheorem{Rem}[Lem]{Remark}

\newtheorem{Expl}[Lem]{Example}

\newenvironment{proof}[1][Proof]{\textrm{\em #1.} }{\hfill $\Box$\medskip\medskip}

\newcommand\Tor{\operatorname{Tor}}

\newcommand\Mon{{\operatorname{Mon}}}

\newcommand\supp{{\operatorname{supp}}}

\newcommand\lex{{\operatorname{lex}}}

\newcommand\pd{{\operatorname{pd}}}
\newcommand\reg{{\operatorname{reg}}}

\newcommand\depth{{\operatorname{depth}}}

\def\NZQ{\mathbb}
\def\NN{{\NZQ N}}
\def\ZZ{{\NZQ Z}}
\def\FF{{\NZQ F}}

\let\emptyset\varnothing

\def\alt{\textup{height}}

\def\p{\mathfrak{p}}
\def\q{\mathfrak{q}}

\begin{document}
\title{\bf\normalsize\MakeUppercase{A note on minimal resolutions of vector-spread Borel ideals}}

\author{Marilena Crupi, Antonino Ficarra}	

\newcommand{\Addresses}{{
\footnotesize
\textsc{Department of Mathematics and Computer Sciences, Physics and Earth Sciences, University of Messina, Viale Ferdinando Stagno d'Alcontres 31, 98166 Messina, Italy}
\begin{center}
 \textit{E-mail addresses}: \texttt{mcrupi@unime.it}; \texttt{antficarra@unime.it}
\end{center}
}}
\date{}
\maketitle
\Addresses

\begin{abstract}
We consider vector--spread Borel ideals. We show that these ideals have linear quotients and thereby we determine the graded Betti numbers and the bigraded Poincar\'e series. A characterization of the extremal Betti numbers of such a class of ideals is given. Finally, we classify all Cohen--Macaulay vector--spread Borel ideals.
\blfootnote{
	\hspace{-0,3cm} \emph{Keywords:} monomial ideals, vector--spread Borel ideals, Cohen--Macaulay ideals, extremal Betti numbers.
	
}

\end{abstract}

\maketitle

\section*{Introduction}
In this note we study the class of vector--spread Borel ideals introduced in \cite{F1} as a generalization of the class of $t$--spread ideals, where $t$ is a non negative integer \cite{EHQ} (see, also, \cite{AFC3, AEL, DHQ} and the reference therein). Let $S=K[x_1,\dots,x_n]$ be the standard graded polynomial ring over a fixed field $K$ and let ${\bf t}=(t_1,t_2,\dots,t_{d-1})\in\ZZ_{\ge0}^{d-1}$, with $d \ge 2$,  a $(d-1)$--tuple whose entries are non negative integers. A monomial $u=x_{j_1}x_{j_2}\cdots x_{j_\ell}$ ($1\le j_1\le j_2\le\cdots\le j_\ell\le n$) of degree $\ell\le d$ of $S$ is called a \textit{vector--spread monomial of type ${\bf t}$} or simply a \textit{${\bf t}$--spread monomial} if $j_{i+1}-j_{i}\ge t_{i}$, for $i=0,\dots,\ell-1$.
A \textit{${\bf t}$--spread monomial ideal} is a monomial ideal generated by ${\bf t}$--spread monomials.
For instance, $I=(x_1x_3^2x_5,x_1x_3^2x_6,x_1x_4x_5)$ is a $(2,0,1)$--spread monomial ideal of the polynomial ring  $S=K[x_1, \ldots, x_5]$, but it is not $(2,1,1)$--spread as $x_1x_3^2x_5\in G(I)$ is not $(2,1,1)$--spread. One can note that any monomial (ideal) is ${\bf 0}$--spread, where ${\bf 0}=(0,0,\dots,0)$. If $t_i\ge1$, for all $i$, a ${\bf t}$--spread monomial (ideal) is a \textit{squarefree} monomial (ideal). A ${\bf t}$--spread monomial ideal $I$ of $S$ is said \emph{vector--spread Borel ideal of type} ${\bf t}$ or simply a \textit{${\bf t}$--spread strongly stable ideal} if for any ${\bf t}$--spread monomial $u\in I$, and all $j<i$ such that $x_i$ divides $u$ and $x_j(u/x_i)$ is ${\bf t}$--spread, then $x_j(u/x_i)\in I$. For ${\bf t}={\bf 0}=(0,0,\dots,0)$ (${\bf t}={\bf 1}=(1,1,\dots,1)$) one obtains the classical notion of strongly stable (squarefree strongly stable) ideal \cite{JT}.\\
We study the graded Betti numbers of a ${\bf t}$--spread strongly stable ideal $I$ by combinatorial tools and therefore we obtain a formula for their computation. As a consequence we are able to give a characterization of the extremal Betti numbers of such an ideal $I$. Finally, the Cohen--Macaulay ${\bf t}$--spread strongly stable ideals are classified. Our approach is similar to the one of \cite{EHQ}. \\
In this note we discuss a unified concept to deal with the classes of monomial ideals which are among the most important in Combinatorics (stable, strongly stable ideals) in order to get more general results.\
Section \ref{sec1} contains some preliminaries and notions that will be used in the article. In Section \ref{sec2}, we review a result on the graded Betti numbers of a ${\bf t}$--spread strongly stable ideal stated in \cite{F1}. We propose a simpler method for determining a formula for the graded Betti numbers of such a class of monomial ideals (Corollary \ref{Cor:BettiNumbFormulaVectSpread}). A key result is Theorem \ref{VectSpreadLinearQuotients} which states that a ${\bf t}$--spread strongly stable ideal has linear quotients. The notion of \textit{vector--spread support} (Subsection 1.2) plays an essential role in this context. Indeed, the linear quotients of vector--spread Borel ideals can be expressed in terms of vector--spread supports (Corollary \ref{cor:set(uk)vect}). 
Moreover, the bigraded Poincar\'e series of such ideals is also determined 
(Corollary \ref{Cor:Poincare}). In Section \ref{sec3}, we obtain a characterization of the extremal Betti numbers of ${\bf t}$--spread strongly stable ideals (Proposition \ref{CharacExtremalBettiVectSpread}) by the results in Section \ref{sec2}. Finally, Section \ref{sec4} contains one of the main result in the article. We analyze the Cohen--Macauleyness of ${\bf t}$--spread strongly stable ideals via the formula of the graded Betti numbers described in Section \ref{sec2}. Indeed, we obtain a classification (Theorem \ref{Thm:CMVectSpread}) by investigating the height and the projective dimension of such ideals.

\section{Preliminaries}\label{sec1}
Throughout the article, we denote by $S=K[x_1,\dots,x_n]$ the standard graded polynomial ring over a field $K$ and by $\Mon(S)$ ($\Mon_d(S)$) the set of all monomials (of degree $d$) in $S$. 
\subsection{A glimpse to graded Betti numbers.} Given a monomial $u=x_{1}^{a_1}x_{2}^{a_2}\cdots x_{n}^{a_n}$ of $S$, 
the \textit{support} of $u$ is the set $\supp(u)=\{i:a_i>0\}=\{i:x_i\ \textup{divides}\ u\}$. 
We denote by $\max(u)$ ($\min(u)$) the maximal (minimal) integer $i\in \supp(u)$. \\
For a monomial ideal $I$ of $S$, $G(I)$ denotes the unique minimal set of monomial generators of $I$ and we set $G(I)_j=\{u\in G(I):\deg(u)=j\}$. It is known that a monomial ideal $I$ of $S$ has a unique minimal graded free $S$--resolution
$$
\FF:0\rightarrow F_p\xrightarrow{\ d_p\ }F_{p-1}\xrightarrow{\ d_{p-1}\ }\cdots\xrightarrow{\ d_1\ }F_0\xrightarrow{\ d_0\ }I\rightarrow0,
$$
where $F_i=\bigoplus_jS(-j)^{\beta_{i,j}(I)}$. 
The numbers $\beta_{i,j}(I)=\dim_K\Tor(K,I)_j$ are called the \textit{graded Betti numbers} of $I$. For all $i$, $\beta_i(I)=\sum_j\beta_{i,j}(I)$ is the \textit{$i$th total Betti number} of $I$. 
One defines the \textit{projective dimension} and the \textit{regularity} of $I$ as follows
\begin{align*}
\pd(I)&=\max\{i:\beta_{i,j}(I)\ne0,\ \textup{for some}\ j\}=\max\{i:\beta_i(I)\ne0\},\\
\reg(I)&=\max\{j-i:\beta_{i,j}(I)\ne0\}=\max\{j:\beta_{i,i+j}(I)\ne0\ \textup{for some}\ i\}.
\end{align*}
These algebraic invariants have been refined in \cite{BCP} by the notion of \emph{extremal Betti number}.
A graded Betti number $\beta_{k,k+\ell}(I)\ne 0$ of a monomial ideal $I$ of $S$ is called \textit{extremal} if $\beta_{i,i+j}(I)=0$ for all $i\ge k$, $j\ge\ell$ such that $(i,j)\ne(k,\ell)$.

Let $I$ be a monomial ideal. It is known that $\textup{depth}(S/I)\le\dim(S/I)$. If the equality holds, we say that $S/I$ is a \textit{Cohen--Macaulay ring} and $I$ is a \textit{Cohen--Macaulay ideal}. 

\subsection{A glimpse to vector--spread monomial ideals.} Let ${\bf t}=(t_1,t_2,\dots,t_{d-1})\in\ZZ_{\ge0}^{d-1}$, $d \ge 2$,  a $(d-1)$--tuple whose entries are non negative integers. Let us write a monomial $u \in S$ as 
$u=x_{j_1}x_{j_2}\cdots x_{j_\ell}$, $1\le j_1\le j_2\le\cdots\le j_\ell\le n$. 
We will maintain this convention throughout the article. 
%
%
%

Let $T=K[x_1,x_2,\dots,x_n,\dots]$ be the polynomial ring in infinitely many variables. Denote by $\Mon(T;{\bf t})$ the set of all ${\bf t}$--spread monomials of $T$ and by $\Mon(S;{\bf t})$ the set of all ${\bf t}$--spread monomials of $S$. Furthermore, for all $0\le\ell\le d$, we define the following sets
\begin{align*}
\Mon_{\ell}(T;{\bf t})&=\big\{u\in\Mon(T;{\bf t}):\deg(u)=\ell\big\},\\
\Mon_{\ell}(S;{\bf t})&=\big\{u\in\Mon(S;{\bf t}):\deg(u)=\ell\big\}.
\end{align*}
Let us denote by $M_{n,\ell,{\bf t}}$ the set of all ${\bf t}$--spread monomials of $S$ having degree $\ell$. One may observe that $M_{n,\ell,{\bf t}}=\emptyset$ for $\ell>d$. 
In order to compute the cardinality of the set $M_{n,\ell,{\bf t}}$, one can consider a special \textit{shifting operator} (see \cite{AHH, EHQ}). More in detail, one defines the map $\sigma_{{\bf 0},{\bf t}}:\Mon(T;{\bf 0})\rightarrow\Mon(T;{\bf t})$, by setting $\sigma_{{\bf 0},{\bf t}}(1)=1$, $\sigma_{{\bf 0},{\bf t}}(x_i)=x_i$ and
\begin{align*}
\sigma_{{\bf 0},{\bf t}}(x_{j_1}x_{j_2}\cdots x_{j_{\ell}})=\prod_{k=1}^{\ell}x_{j_k+\sum_{s=1}^{k-1}t_s},
\end{align*}
for all monomials $u=x_{j_1}x_{j_2}\cdots x_{j_{\ell}}\in\Mon(T;{\bf 0})$ with $2\le\ell\le d$.
The map $\sigma_{{\bf 0},{\bf t}}$ is bijective and its inverse is the map $\sigma_{{\bf t},{\bf 0}}: \Mon(T;{\bf t})\rightarrow\Mon(T;{\bf 0})$ defined as follows: $\sigma_{{\bf t},{\bf 0}}(1)=1$, $\sigma_{{\bf t},{\bf 0}}(x_i)=x_i$, for all $i\in\NN$, and $\sigma_{{\bf t},{\bf 0}}(x_{j_1}x_{j_2}\cdots x_{j_{\ell}})$ $=\prod_{k=1}^{\ell}x_{j_k-\sum_{s=1}^{k-1}t_s}$, for all monomials $u=x_{j_1}x_{j_2}\cdots x_{j_{\ell}}\in\Mon(T;{\bf t})$ with 
$2\le\ell\le d$. \\
%
In particular, the restriction $\sigma_{{\bf t},{\bf 0}}|_{M_{n,\ell,{\bf t}}}$ is a injective map and its image is equal to the set $M_{n-(t_1+t_2+\ldots+t_{\ell-1}),\ell,{\bf 0}}=\Mon_{\ell}(K[x_1,\dots,x_{n-(t_1+t_2+\ldots+t_{\ell-1})}])$. Thus
\begin{equation}\label{eq:MnellVectCard}
|M_{n,\ell,{\bf t}}|=\binom{n+(\ell-1)-\sum_{j=1}^{\ell-1}t_j}{\ell}, \quad \mbox{for $0\le\ell\le d$}.
\end{equation}
For more details on this topic see \cite{F1}.


If $u_1,\dots,u_m$ are ${\bf t}$--spread monomials of $S$, we denote by $B_{\bf t}(u_1,\dots,u_m)$ the smallest ${\bf t}$--spread strongly stable ideal of $S$ containing $u_1,\dots,u_m$ with respect to the inclusion. The monomials $u_1,\dots,u_m$ are called the \textit{${\bf t}$--spread Borel generators} of $B_{\bf t}(u_1,\dots,u_m)$. If $m=1$, $B_{\bf t}(u_1)=B_{\bf t}(u)$ is called a \textit{principal ${\bf t}$--spread Borel ideal}. It is clear that for each ${\bf t}$--spread strongly stable monomial ideal $I$ of $S$, one can uniquely determine ${\bf t}$--spread monomials $u_1,\dots,u_m\in S$ such that $I=B_{\bf t}(u_1,\dots,u_m)$. \\
Observe that for $u=x_{n-(t_{1}+t_2+\ldots+t_{\ell-1})}x_{n-(t_2+\ldots+t_{\ell-1})}\cdots x_{n-t_{\ell-1}}x_{n}$, the ideal $B_{\bf t}(u)$ contains all ${\bf t}$--spread monomials of degree $\ell\le d$. $B_{\bf t}(u)$ is called the \textit{${\bf t}$--spread Veronese ideal of degree $\ell$} and will be denoted by $I_{n,\ell,{\bf t}}$. Note that 
$I_{n,\ell,{\bf t}}\ne(0)$ if and only if $n\ge 1+\sum_{j=1}^{\ell-1}t_j$.\\

Now let us recall the pivotal notion of \textit{${\bf t}$--spread support} of a ${\bf t}$--spread monomial \cite[Definition 2.1]{F1}.

If $n$ is a positive integer $n$, $[n]$ denotes the set $\{1,2,\dots,n\}$. 
For $j \le k$ positive integers, we define $[j,k] = \{\ell\in\mathbb{N}:j\le\ell\le k\}$. 
We set $[j,k]= \emptyset$ if $j>k$.\\
Let $u=x_{j_1}x_{j_2}\cdots x_{j_{\ell}}\in M_{n,\ell,{\bf t}}$ be a ${\bf t}$--spread monomial of $S$. The \textit{${\bf t}$--spread support} of $u$ is the subset of $[n]$ defined as follows
$$
\supp_{\bf t}(u)=\textstyle\bigcup\limits_{i=1}^{\ell-1}\ \big[j_i,j_i+(t_i-1)\big].
$$
Note that $\supp_{\bf 0}(u)=\emptyset$. Furthermore, if $u$ is squarefree, $\supp_{\bf 1}(u)$ $=\supp(u/x_{\max(u)})$ $=\big\{j_1,j_2,\dots,j_{\ell-1}\big\}$, where ${\bf 1}=(1,1,\dots,1)\in\ZZ_{\ge0}^{d-1}$.

The notion vector--spread support has been introduced in \cite{CF1} in order to classify Cohen--Macaulay $t$--spread lexsegment ideals.


\section{The graded Betti numbers}\label{sec2}
In \cite{F1}, the graded Betti numbers of vector--spread strongly stable ideals have been determined by means of the Koszul homology. More precisely, if $I$ is a ${\bf t}$--spread strongly stable ideal of $S$, then the graded Betti numbers of $I$ are given by
$$
\beta_{i,i+j}(I)=\sum_{u\in G(I)_j}\binom{\max(u)-1-\sum_{\ell=1}^{j-1}t_\ell}{i},\quad \mbox{for all $i, j$}.
$$
The purpose of this section is to quickly obtain this formula 
using a different and easier method than the one used in \cite{F1}. The vector--spread support will be a key tool.

We need to fix some notations and recall some results from \cite{ET}. Let $I$ be a ${\bf t}$--spread strongly stable ideal of $S$. Let $G(I)=\{u_1>u_2>\dots>u_m\}$ be the minimal generating set of $I$ ordered with respect to the \textit{pure lexicographic order} \cite{JT}. For $k=1,\dots,m$, we set $J_k=(u_1,\dots,u_{k})$ and
$$
\textup{set}(u_k) = \big\{i\in[n]:x_i\in(u_1,\dots,u_{k-1}):u_k \big\}.
$$
Note that $\textup{set}(u_1)=\emptyset$. Our aim is to prove that $I$ has \textit{linear quotients}, \emph{i.e.}, 
$$
J_{k-1}:u_k=(u_1,u_2,\dots,u_{k-1}):u_k,
$$
is generated by variables, for all $k=2,\dots,m$. In such a case, by \cite[Lemma 1.5]{ET}, one has
\begin{equation}\label{eq:ETlinquot}
\beta_{i,i+j}(I)=\sum_{u\in G(I)_j}\binom{|\textup{set}(u)|}{i}, \quad \mbox{for all $i,j\ge0$}.
\end{equation}

We quote the next crucial lemma from \cite{F1} (see also \cite[Lemma 1.3]{EHQ}). 
It provides the existence of a \textit{standard decomposition} for a ${\bf t}$--spread monomial belonging to a ${\bf t}$--spread strongly stable ideal.
\begin{Lem}\label{LemG(I)M(I)}
	Let $I$ be a ${\bf t}$--spread strongly stable ideal of $S$, and $w\in I$ a ${\bf t}$--spread monomial. Then, there exist $u\in G(I)$ and $v\in\Mon(S)$ such that $w=uv$ and $\max(u)\le\min(v)$.
\end{Lem}

Now we are in position to prove the main result of this section.
\begin{Thm}\label{VectSpreadLinearQuotients}
	Let $I$ be a ${\bf t}$--spread strongly stable ideal of $S$. Then $I$ has linear quotients, in particular it is componentwise linear.
\end{Thm}
\begin{proof}
	Let $G(I)=\{u_1>u_2>\dots>u_m\}$ ordered with respect to the pure lexicographic order. Fix $k\in\{2,\dots,m\}$. A set of generators for $J_{k-1}:u_k=(u_1,\dots,u_{k-1}):u_k$ is $\{u_s/\textup{gcd}(u_k,u_s):s=1,\ldots,k-1\}$ \cite[Proposition 1.2.2]{JT}. Hence, it suffices to show that for each $s$, there exists a variable $x_i\in J_{k-1}:u_k$ with $x_i$ dividing $u_s/\textup{gcd}(u_k,u_s)$. Let $u_s=x_{i_1}x_{i_2}\cdots x_{i_p},\ u_k=x_{j_1}x_{j_2}\cdots x_{j_q},$ with $u_s>_{\lex}u_k$. Hence, there exists an integer $\ell\in[q]$ such that
	\begin{equation}\label{eq:deflexord}
	i_1=j_1,\ \ i_2=j_2,\ \ \ldots,\ \ i_{\ell-1}=j_{\ell-1},\ \ i_\ell<j_\ell.
	\end{equation}
	Set $v=x_{i_\ell}(u_k/x_{j_\ell})$. Then $v=x_{j_1}x_{j_2}\cdots x_{j_{\ell-1}}x_{i_\ell}x_{j_{\ell+1}}\cdots x_{j_q}$. We need to show that $v$ is ${\bf t}$--spread. This is the case, as $i_{\ell}-j_{\ell-1}=i_{\ell}-i_{\ell-1}\ge t_{\ell-1}$ and $j_{\ell+1}-i_\ell>j_{\ell+1}-j_{\ell}\ge t_{\ell}$.
	On the other hand, $I$ is a ${\bf t}$--spread strongly stable ideal. Therefore $v\in I$ and, furthermore, $v>_{\lex}u_k$. Now we show that $v\in J_{k-1}$. By Lemma \ref{LemG(I)M(I)}, $v=u_{r}w$ with $u_r\in G(I)$ and $w\in\Mon(S)$ such that $\max(u_r)\le\min(w)$. If $v\notin J_{k-1}=(u_1,\dots,u_{k-1})$, then $u_r\in G(I)\setminus\{u_1,\dots,u_{k-1}\}$, and consequently $u_r\le_{\lex}u_k$. As $v=u_rw$, then $v\le_{\lex}u_k$, a contradiction.
	Therefore, $v\in J_{k-1}$ and $x_{i_\ell}u_k=vx_{j_\ell}\in J_{k-1}$. Thus $x_{i_\ell}\in J_{k-1}:u_k$ and, by (\ref{eq:deflexord}), $x_{i_\ell}$ divides $u_s/\textup{gcd}(u_k,u_s)$. It follows that $I$ has linear quotients. Finally, it is well known that ideals with linear quotients are componentwise linear \cite[Theorem 8.2.15]{JT}.
\end{proof}

The following corollary highlights the role of the vector--spread supports.
\begin{Cor}\label{cor:set(uk)vect}
	In the previous setting, for all $k=2,\dots,m$, we have
	$$
	\textup{set}(u_k)=[\max(u_k)-1]\setminus\supp_{\bf t}(u_k).
	$$
\end{Cor}
\begin{proof}
	Let $u_k=x_{j_1}x_{j_2}\cdots x_{j_d}$. We have $\supp_{\bf t}(u_k)=\bigcup_{i=1}^{d-1}\big(\bigcup_{q=0}^{t_i-1}\{j_i+q\}\big)$. Let $\ell\in[n]$. To determine if $\ell\in\textup{set}(u_k)$, we consider two cases.\\
	\textsc{Case 1}. Let $\max(u_k)\le\ell\le n$, then $x_\ell u_k=x_{j_1}x_{j_2}\cdots x_{j_d}x_\ell$. First, observe that the ideal $J_{k-1}=(u_1,u_2,\dots,u_{k-1})$ is ${\bf t}$--spread strongly stable. If for absurd $\ell\in\textup{set}(u_k)$, then $x_\ell u_k\in J_{k-1}$. Hence $u_r$ divides $x_\ell u_k$, for some $r\in\{1,\dots,k-1\}$. Since $\deg(u_r)\le\deg(u_k)$, then $u_r$ divides $u_k$. An absurd, since $u_k$ is a minimal generator of $I$.\\
	\textsc{Case 2}. Let $\ell\in\supp_{\bf t}(u_k)$, then $\ell=j_r+s$, for some $r\in\{1,\dots,d-1\}$, $0\le s\le t_r-1$, $t_r\ge 1$. If $x_\ell u_k\in J_k=(u_1,u_2,\dots,u_k)$, then $u_p$ divides $x_\ell u_k$, for some $p\in\{1,\dots,k-1\}$. As $j_r+s-j_r=s<t_r$, and $u_p$ is ${\bf t}$--spread, necessarily $u_p$ divides $u_k$. An absurd, since $u_k$ is a minimal generator of $I$.
	
	The above cases imply that if $\ell\in\textup{set}(u_k)$, then $\ell \notin [\max(u_k), n]$ and $\ell \notin \supp_{\bf t}(u_k)$.
	
	Now assume $\ell\in[\max(u_k)-1]\setminus\supp_{\bf t}(u_k)$. Then $\ell=j_r+s$, for some $r\in\{1,\dots,d-1\}$ and $s\ge t_r$. Let $j_q=\min\big\{j\in\supp(u_k):j>\ell\big\}$. Then $x_\ell(u_k/x_{j_{q}})$ is a ${\bf t}$--spread monomial that belongs to the ideal $(J_{k-1},u_k)=(u_1,u_2,\dots,u_{k-1},u_k)=J_k$, as $\ell<j_q$ and $(J_{k-1},u_k)$ is a ${\bf t}$--spread strongly stable ideal. By Lemma \ref{LemG(I)M(I)}, $x_\ell(u_k/x_{j_q})=u_pw$ with $p\in\{1,\dots,k\}$ and $w\in\Mon(S)$ such that $\max(u_p)\le\min(w)$. Clearly, $p\ne k$ and so $x_\ell(u_k/x_{j_q})\in J_{k-1}=(u_1,u_2,\dots,u_{k-1})$. Finally $x_\ell u_k=x_\ell(u_k/x_{j_q})x_{j_q}\in J_{k-1}$ and $\ell\in\textup{set}(u_k)$. The assertion follows.
\end{proof}

As a consequence of the above result, one has
$$
J_{k-1}:u_k=\big(x_i:i\in[\max(u_k)-1]\setminus\supp_{\bf t}(u_k)\big),\quad  \mbox{for $k=2,\dots,m$.}
$$
On the other hand, for each $k$ 
\[\big|[\max(u_k)-1]\setminus\supp_{\bf t}(u_k)\big|=\max(u_k)-1-\sum_{\ell=1}^{\deg(u_k)-1}t_\ell.\] 
Thus, formula (\ref{eq:ETlinquot}) yields the result we are looking for.
\begin{Cor}\label{Cor:BettiNumbFormulaVectSpread}
	Let $I$ be a ${\bf t}$--spread strongly stable ideal of $S$. Then, 
	$$
	\beta_{i,i+j}(I)=\sum_{u\in G(I)_j}\binom{\max(u)-1-\sum_{\ell=1}^{j-1}t_\ell}{i}, \quad \mbox{for all $i,j\ge0$}.
	$$
\end{Cor}
As a consequence, we obtain:

\begin{Cor} \label{Cor:Poincare}
	Let $I$ be a ${\bf t}$--spread strongly stable ideal of $S$. Then
	$$
	P_{S/I}(y,z)=1+\sum_{u\in G(I)}(1+y)^{\max(u)-1-\sum_{s=1}^{\deg(u)-1}t_s}yz^{\deg(u)}.
	$$
\end{Cor}
\begin{proof} Let $P_{S/I}(y,z)=\sum_{i,j}\beta_{i,j}(I)y^iz^j$ be the bigraded Poincar\'e series of $S/I$. From \cite[Corollary 1.6]{ET}, for an ideal $I$ with linear quotients, one has
	$$
	P_{S/I}(y,z)=1+\sum_{u\in G(I)}(1+y)^{|\textup{set}(u)|}yz^{\deg(u)}.
	$$
	The assertion follows from Corollary \ref{Cor:BettiNumbFormulaVectSpread}.
\end{proof}
\section{Extremal Betti numbers}\label{sec3}
In this section we characterize the extremal Betti numbers of the class of ${\bf t}$--spread Borel ideals. Our results generalizes some statements in \cite{AC2} (see also \cite{AFC1, AFC2, MC, CF} and the reference therein). \\

First, one can note that Corollary \ref{Cor:BettiNumbFormulaVectSpread} implies 
\begin{align}
\label{eq:regVect}\reg(I)\ &=\ \max\big\{\deg(u):u\in G(I)\big\},\\
\label{eq:pdVect}\pd(I)\ &=\ \max\big\{\max(u)-1-\textstyle\sum_{j=1}^{\deg(u)-1}t_j:u\in G(I)\big\}.
\end{align}

For our aim, the next lemma will play a crucial role.
\begin{Lem}
	Let $I$ be a ${\bf t}$--spread strongly stable ideal of $S$. If $\beta_{i,i+j}(I)\ne0$, then $\beta_{k,k+j}(I)\ne0$ for all $k=0,\dots,i$.
\end{Lem}
\begin{proof}
	Let $\beta_{i,i+j}(I)\ne0$. From Corollary \ref{Cor:BettiNumbFormulaVectSpread}, there exists a monomial $u\in G(I)_j$ such that $\max(u)-1-\sum_{\ell=1}^{j-1}t_\ell\ge i$. Hence, $\max(u)-1-\sum_{\ell=1}^{j-1}t_\ell\ge k$, for all $k=0,\dots,i$. This implies that $\beta_{k,k+j}(I)\ne0$, for all $k=0,\dots,i$.
\end{proof}

The previous lemma and the definition of extremal Betti number imply the next result.
\begin{Cor}\label{Cor:extremBettNumbVectT}
	Let $I$ be a ${\bf t}$--spread strongly stable ideal of $S$. The following conditions are equivalent:
	\begin{enumerate}
		\item[\textup{(i)}] $\beta_{k,k+\ell}(I)$ is extremal;
		\item[\textup{(ii)}] $\beta_{k,k+\ell}(I)\ne0$, and $\beta_{i,i+\ell}(I),\beta_{k,k+j}(I)=0$, for all $i>k$ and all $j>\ell$.
	\end{enumerate}
\end{Cor}
The following characterization generalizes a result in \cite[Theorem 1]{AC2}.
\begin{Prop}\label{CharacExtremalBettiVectSpread}
	Let $I$ be a ${\bf t}$--spread strongly stable ideal of $S$. The following conditions are equivalent:
	\begin{enumerate}
		\item[\textup{(i)}] $\beta_{k,k+\ell}(I)$ is extremal;
		\item[\textup{(ii)}] $\max\big\{\max(u):u\in G(I)_\ell\big\}=k+\sum_{j=1}^{\ell-1}t_j+1$, and $\max(u)<k+\sum_{r=1}^{j-1}t_r+1$, for all $u\in G(I)_j$ and all $j>\ell$.
	\end{enumerate}
\end{Prop}
\begin{proof}
	(i)$\implies$(ii). Since $I$ is ${\bf t}$--spread strongly stable, Corollary \ref{Cor:BettiNumbFormulaVectSpread} implies that $\beta_{k,k+\ell}(I)$ is non zero if and only if there exists a monomial $u_0\in G(I)_\ell$ such that $\max(u_0)\ge k+\sum_{j=1}^{\ell-1}t_j+1$. Therefore, 
	\begin{equation}\label{eq:extrBettiVectSpread}
	\max\big\{\max(u):u\in G(I)_\ell\big\}\ \ge\ \max(u_0)\ \ge\ k+\textstyle\sum_{j=1}^{\ell-1}t_j+1.
	\end{equation}
	Suppose that for a monomial $u_1\in G(I)_\ell$, $\max(u_1)=j+\sum_{j=1}^{\ell-1}t_j+1>k+\sum_{j=1}^{\ell-1}t_j+1$, for some $j>k$. Then $\beta_{j,j+\ell}(I)\ne 0$, which contradicts the fact that $\beta_{k,k+\ell}(I)$ is extremal (Corollary \ref{Cor:extremBettNumbVectT}). Hence, $k+\sum_{j=1}^{\ell-1}t_j+1= \max\big\{\max(u):u\in G(I)_\ell\big\}$.
	Assume there exists a monomial $v\in G(I)_j$ such that $\max(v)\ge k+\sum_{r=1}^{j-1}t_r+1$, for some $j>\ell$. Then $\beta_{k,k+j}(I)\ne0$. A contradiction, since $\beta_{k,k+\ell}(I)$ is extremal (Corollary \ref{Cor:extremBettNumbVectT}). Hence, condition (ii) holds.\\
	(ii)$\implies$(i). Since $\max\big\{\max(u):u\in G(I)_\ell\big\}=k+\sum_{j=1}^{\ell-1}t_j+1$, then $\beta_{k,k+\ell}(I)\ne0$ and $\beta_{j,j+\ell}(I)=0$, for all $j>k$. Moreover, $\max(u)<k+\sum_{r=1}^{j-1}t_r+1$, for all $u\in G(I)_j$ and all $j>\ell$, implies that $\beta_{k,k+j}(I)=0$, for all $j>\ell$. Finally, by Corollary \ref{Cor:extremBettNumbVectT}, $\beta_{k,k+\ell}(I)$ is extremal. 
\end{proof}

The previous result yields the following useful corollary.
\begin{Cor}\label{Cor:ExtrBettiVectSpreadConto}
	Let $I$ be a ${\bf t}$--spread strongly stable ideal of $S$, and let $\beta_{k,k+\ell}(I)$ be an extremal Betti number of $I$. Then
	\[\beta_{k,k+\ell}(I) = \big|\big\{u\in G(I)_\ell:\max(u)=k+\textstyle\sum_{j=1}^{\ell-1}t_j+1\big\}\big| \le\ \tbinom{k+\ell-1}{\ell-1}.\]
\end{Cor}
\begin{proof} The equality follows immediately from Proposition \ref{CharacExtremalBettiVectSpread}. For the inequality, it sufficies to observe that
	\begin{align*}
	&\big|\big\{u\in M_{n,\ell,{\bf t}}:\max(u)=k+\sum_{j=1}^{\ell-1}t_j+1\big\}\big| =\\ & =\binom{k+\sum_{j=1}^{\ell-1}t_j+1+(\ell-1)-\sum_{j=1}^{\ell-1}t_j-1}{\ell -1} = \binom{k+\ell-1}{\ell-1}.
	\end{align*}
\end{proof}

\begin{Expl}\label{Ex:KoszVect2}
	\rm Let $I=(x_1x_2,x_1x_3,x_1x_4^2)$ be a $(1,0)$--spread strongly stable ideal of $S=K[x_1,x_2,x_3,x_4]$
	
	The Betti table of $I$ is \cite{GDS}:
	$$
	\begin{matrix}
	&0&1&2\\
	\text{total:}&3&3&1\\
	\text{2:}&2&1&\text{-}\\
	\text{3:}&1&2&1
	\end{matrix}
	$$
	Hence, $\pd(I)=2$, $\reg(I)=3$. 
	Moreover, $\beta_{2,2+3}(I)$ is the unique extremal Betti number of $I$ and by Corollary \ref{Cor:ExtrBettiVectSpreadConto},
	\[
	\beta_{2,2+3}(I) = \big|\big\{u\in G(I)_3:\max(u)=2+\textstyle\sum_{j=1}^{3-1}t_j+1=4\big\}\big|=\big|\big\{x_1x_4^2\big\}\big| = 1. 
	\]
	
\end{Expl}

\section{Cohen--Macaulay vector--spread Borel ideals}\label{sec4}

In this section we investigate the Cohen--Macaulayness of vector--spread Borel ideals by tools from \cite{CF1, EHQ}. To obtain a characterization, we only need to investigate the height of a vector--spread Borel ideal. Indeed, from Corollary \ref{Cor:BettiNumbFormulaVectSpread} and the Auslander--Buchsbaum formula, one has that
$$\depth(I) =n-\max\big\{\max(u)-1-\textstyle\sum_{j=1}^{\deg(u)-1}t_j:u\in G(I)\big\}.$$

\begin{Prop}\label{CorAltVerVectSpr}
	Let $I_{n,d,{\bf t}}$ be the ${\bf t}$--spread Veronese ideal of degree $d$ of $S$. Then $$\textup{height}(I_{n,d,{\bf t}})=n-\textstyle\sum_{j=1}^{d-1}t_j.$$
\end{Prop}
\begin{proof} 
	First, observe that $I_{n,d,{\bf t}}\ne(0)$ if and only if $n\ge 1+\sum_{j=1}^{d-1}t_j$.
	The monomial prime ideal $\mathfrak{p}=(x_i:i=1,2,\dots,n-\sum_{j=1}^{d-1}t_j)$ is a minimal prime of $I_{n,d,{\bf t}}$, as each monomial generator $u\in G(I_{n,d,{\bf t}})=M_{n,d,{\bf t}}$ has minimum belonging to the set $\big\{1,2,\dots,n-\sum_{j=1}^{d-1}t_j\big\}$. Therefore $$\textup{height}(I_{n,d,{\bf t}})\le n-\textstyle\sum_{j=1}^{d-1}t_j.$$
	Suppose $\textup{height}(I_{n,d,{\bf t}})=k<n-\sum_{j=1}^{d-1}t_j$. Then, there would be a minimal prime $\mathfrak{q}=(x_{i_j}:j=1,\dots,k)$ of $I_{n,d,{\bf t}}$ of height $k\le n-\sum_{j=1}^{d-1}t_j-1$, with all $i_j$ distinct. The set $A=[n]\setminus\{i_j:j=1,\dots,k\}$ has cardinality
	$$
	|A|=n-k\ge n-\big(n-\textstyle\sum_{j=1}^{d-1}t_j-1\big)=1+\textstyle\sum_{j=1}^{d-1}t_j.
	$$
	We set $\ell_1=\min(A)$ and $\ell_j=\min\{\ell\in A:\ell\ge\ell_{j-1}+t_j\}$, for all $j\ge2$. The sequence $\ell_1\le\ell_2\le\dots\le\ell_k$ has at least $d$ terms, otherwise $|A|<1+\textstyle\sum_{j=1}^{d-1}t_j$. So, the monomial $w=x_{\ell_1}x_{\ell_2}\cdots x_{\ell_d}$ is ${\bf t}$--spread of degree $d$, but $w\in G(I_{n,d,{\bf t}})\setminus\q$, an absurd. Finally, we have $\textup{height}(I_{n,d,{\bf t}})=n-\sum_{j=1}^{d-1}t_j$, as desired.
\end{proof}

\begin{Thm}\label{thm:heightVectSpread}
	Let $I$ be a ${\bf t}$--spread strongly stable ideal of $S$. Then 
	$$
	\alt(I)=\max\big\{\min(u):u\in G(I)\big\}.
	$$
\end{Thm}
\begin{proof} 
	First we prove that $\alt(I)\le\max\{\min(u):u\in G(I)\}$. Indeed, let $u_0\in G(I)$ such that $\min(u_0)=\max\{\min(u):u\in G(I)\}$. In such a case the monomial prime ideal $\p_{[\min(u_0)]}=(x_1,x_2,\dots,x_{\min(u_0)})$ is a minimal prime of $I$, thus proving the inequality.
	For the other inequality, write $u_0=x_{j_1}x_{j_2}\cdots x_{j_d}$ and consider the monomial $v_0=x_{j_1}x_{j_1+t_1}\cdots x_{j_1+(t_1+t_2+\ldots+t_{d-1})}$. Then $v_0\in B_{\bf t}(u_0)\subseteq I$. Moreover, $I'=B_{\bf t}(v_0)=I_{j_1+t_1+t_2+\ldots+t_{d-1},d,{\bf t}}$. Hence, by Proposition \ref{CorAltVerVectSpr}
	\[\alt(I)\ge\alt(I')=j_1+\sum_{\ell=1}^{d-1}t_\ell-\sum_{\ell=1}^{d-1}t_\ell=j_1=\min(v_0)=\min(u_0).\]
\end{proof}

Using what we have shown thus far, we are able to classify all Cohen--Macaulay vector--spread strongly stable ideals of $S$.
\begin{Thm}\label{Thm:CMVectSpread}
	Let $I\subseteq S$ be a ${\bf t}$--spread strongly stable ideal of $S$, ${\bf t}=(t_1,t_2,\dots,t_{d-1})$. 
	Then $S/I$ is a Cohen--Macaulay ring if and only if there exists $u\in G(I)$ of degree $\ell \le d$ such that $$u=x_{n-(t_1+t_2+\ldots+t_{\ell-1})}x_{n-(t_2+t_3+\ldots+t_{\ell-1})}\cdots x_{n-t_{\ell-1}}x_n\in G(I).$$
	In particular, if $I$ is equigenerated in degree $\ell$ then $S/I$ is Cohen--Macaulay if and only if $I=I_{n,\ell,{\bf t}}$ is a vector--spread Veronese ideal.
\end{Thm}
\begin{proof}
	Firstly, we may suppose $n\in\bigcup_{u\in G(I)}\supp(u)$. On the contrary, setting $\widetilde{n}=$ $\max\bigcup_{u\in G(I)}\supp(u)$, we may replace $S$ with $\widetilde{S}=K[x_1,\dots,x_{\widetilde{n}}]$ and consider $I\cap\widetilde{S}$, instead. Since $\pd(S/I)=\pd(I)+1$, by (\ref{eq:pdVect}), we have
	\begin{align*}
	\pd(S/I)\ &=\ {\textstyle\max\big\{\max(u)-\sum_{j=1}^{\deg(u)-1}t_j:u\in G(I)\big\}}.
	\intertext{\noindent On the other hand, by Theorem \ref{thm:heightVectSpread},}
	\dim(S/I)\ &=\ n-\max\big\{\min(u):u\in G(I)\big\}.
	\end{align*}
	By the Auslander--Buchsbaum formula, $n-\pd(S/I)=\depth(S/I)$. Since $S/I$ is Cohen--Macaulay if and only if $\depth(S/I)=\dim(S/I)$, by the previous formulas, $S/I$ is Cohen--Macaulay if and only if
	\begin{equation}\label{eq:CHVectSpr}
	\max\Big\{\max(u)-\sum_{j=1}^{\deg(u)-1}t_j:u\in G(I)\Big\}=\max\big\{\min(u):u\in G(I)\big\}.
	\end{equation}
	Let $u_0\in G(I)$ such that $\min(u_0)=\max\big\{\min(u):u\in G(I)\big\}$. Clearly, for all $u\in G(I)$, $\min(u)\le\max(u)-\sum_{j=1}^{\deg(u)-1}t_j$. Hence, equation (\ref{eq:CHVectSpr}) holds if and only if
	$$
	\min(u_0)=\max(u_0)-\textstyle\sum_{j=1}^{\deg(u)-1}t_j.
	$$
	Thus, $S/I$ is Cohen--Macaulay if and only if there exist a monomial $u_0\in G(I)$ and an integer $\ell\le d$ such that
	$$
	\min(u_0)=\max\big\{\min(u):u\in G(I)\big\},\ u_0=x_{j_1}x_{j_1+t_1}\cdots x_{j_1+(t_1+t_2+\ldots+t_{\ell-1})}.
	$$
	It remains to show that $j_1=n-\sum_{r=1}^{\ell-1}t_r$. Since $n\in\bigcup_{u\in G(I)}\supp(u)$, it follows the existence of a monomial $u\in G(I)$ with $\max(u)=n$. Moreover, $\min(u)\le\min(u_0)$, as $\min(u_0)=\max\{\min(u):u\in G(I)\}$. Note that 
	$$
	\max(u)-\textstyle\sum_{r=1}^{\deg(u)-1}t_r\le j_1=\min(u_0).
	$$
	Hence, as $\max(u)=n$, 
	we have $j_1\ge n-\sum_{r=1}^{\deg(u)-1}t_r$. If $\deg(u)\le\deg(u_0)$, then $j_1\ge n-\sum_{r=1}^{\deg(u)-1}t_r\ge n-\sum_{r=1}^{\deg(u_0)-1}t_r$. On the other hand, $j_1\le n-\sum_{r=1}^{\deg(u_0)-1}t_r$. So, in such a case, $j_1=n-\sum_{r=1}^{\deg(u_0)-1}t_r$, as desired.\\
	Suppose now $\deg(u)>\deg(u_0)$. If $u=x_{k_1}x_{k_2}\cdots x_{k_\ell}x_{k_{\ell+1}}\cdots x_{k_{\deg(u)}}$, where $\ell=\deg(u_0)$, we set $u_1=x_{k_1}x_{k_2}\cdots x_{k_\ell}$. If $k_\ell\le j_1+\sum_{r=1}^{\ell-1}t_r$, then $u_1$ is ${\bf t}$--spread and $k_1\le j_1,\ k_2\le j_2,\ \dots,\ k_\ell\le j_{\ell}$, where $j_i=j_1+\sum_{j=1}^{i-1}t_j$ ($i=1, \ldots, \ell$). Since $I$ is ${\bf t}$--spread strongly stable, then $u_1\in I$.  A contradiction, as $u_1$ divides $u$ and $u$ is a minimal generator. Therefore, we must have $k_\ell>j_1+\sum_{r=1}^{\ell-1}t_r$ and consequently 
	$$
	\max(u)-\sum_{r=1}^{\deg(u)-1}t_r\ge\max(u_1)-\sum_{r=1}^{\deg(u_1)-1}t_r>j_1,
	$$
	which contradicts (\ref{eq:CHVectSpr}). 
\end{proof}
\begin{Rem}\rm
	From Theorem \ref{Thm:CMVectSpread}, the vector--spread Veronese ideal $I_{n,d,{\bf t}}$ is Cohen--Macaulay. Moreover, $\pd(S/I_{n,d,{\bf t}})=n-\sum_{j=1}^{d-1}t_j$. So $I_{n,d,{\bf t}}$ is a Cohen--Macaulay ideal with pure resolution of type $(d_1,\dots,d_p)$, with $d_j=d+j-1$, for $j=1,\dots,p$. Therefore, by \cite[Theorem 4.1.15]{BH98}, we have for all $i\ge1$,
	\begin{align*}
	\beta_{i}(S/I_{n,d,{\bf t}})\ &=\ (-1)^{i+1} \prod_{j \neq i} \frac{d_{j}}{d_{j}-d_{i}}= (-1)^{i+1}\prod_{j=1}^{i-1}\frac{d+j-1}{j-i}\prod_{j=i+1}^{p} \frac{d+j-1}{j-i}\\
	&=\ \binom{d+i-2}{d-1}\binom{n-\sum_{j=1}^{d-1}t_j+d-1}{d+i-1}.
	\end{align*}
	Note that for $i=1$, we obtain $\mu(I_{n,d,{\bf t}})=\beta_1(S/I_{n,d,{\bf t}})=\binom{n-\sum_{j=1}^{d-1}t_j+d-1}{d}=|M_{n,d,{\bf t}}|$ and we get again formula (\ref{eq:MnellVectCard}). 
\end{Rem}

\end{document}